\documentclass[11pt, reqno]{article}

\usepackage{fullpage}
\usepackage{enumerate}
\usepackage{setspace}

\usepackage{upgreek, mathtools,bm}

\usepackage{amsmath,amsfonts,amssymb,graphics,amsthm, comment}
\usepackage{hyperref}
\hypersetup{
    colorlinks=true,
    linkcolor=blue,
    citecolor=red,
    urlcolor=blue,
    pdfborder={0 0 0}
}
\usepackage{yfonts}

\usepackage{graphicx}
\usepackage{xcolor}
\usepackage{bbm}
\usepackage{wrapfig} % for figures inserted on the side of the text

\usepackage[font=sf, labelfont={sf,bf}, margin=1cm]{caption}

\usepackage{cleveref}
  \crefname{theorem}{Theorem}{Theorems}
  \crefname{thm}{Theorem}{Theorems}
  \crefname{thm*}{Theorem*}{Theorems}
  \crefname{lemma}{Lemma}{Lemmas}
  \crefname{lem}{Lemma}{Lemmas}
  \crefname{remark}{Remark}{Remarks}
  \crefname{prop}{Proposition}{Propositions}
\crefname{notation}{Notation}{Notations}
\crefname{claim}{Claim}{Claims}
  \crefname{defn}{Definition}{Definitions}
  \crefname{corollary}{Corollary}{Corollaries}
  \crefname{section}{Section}{Sections}
  \crefname{figure}{Figure}{Figures}
    \crefname{assumption}{Assumption}{Assumptions}

\newtheorem{thm}{Theorem}[section]
\newtheorem{thm*}{Theorem*}[section]

\newtheorem{lemma}[thm]{Lemma}
\newtheorem{corollary}[thm]{Corollary}

\newtheorem{defn}[thm]{Definition}

\numberwithin{equation}{section}

\theoremstyle{definition}
\newtheorem{remark}[thm]{Remark}

\def\cC{\mathcal{C}}

\def\P{\mathbb{P}}

\def\E{\mathbb{E}}

\def\R{\mathbb{R}}

\def\N{\mathbb{N}}

\def  \p- {p\textunderscore}

\def\ep{\varepsilon}

\newcommand{\bJ}{\mathbf{J}}

\newcommand{\EX}[1]{\E\left[#1\right]}
\newcommand{\al}{\alpha}
\newcommand{\be}{\beta}
\newcommand{\ga}{\gamma}
\newcommand{\de}{\delta}
\newcommand{\si}{\sigma}

\newcommand{\lam}{\lambda}

\newcommand{\iy}{\infty}

\newcommand{\bS}{\mathbb{S}}
\newcommand{\na}{\nabla}
\newcommand{\ct}{\cdot}

\newcommand{\lr}[1]{\left( {#1}\right)}

\newcommand{\la}[1]{ \left\lvert {#1}\right\rvert}

\newcommand{\rc}[1]{\frac{1}{#1}}

\newcommand{\fc}[2]{\frac{#1}{#2}}

\newcommand{\sq}{\sqrt}

\newcommand{\sumiN}{\sum_{i=1}^N}
\newcommand{\sumjN}{\sum_{j=1}^N}

\newcommand{\1}{\mathbf{1}}

\newcommand{\Ls}{\mathcal{L}}

\begin{document}
 \title{Analyzing dynamics and average case complexity in the spherical Sherrington-Kirkpatrick model: a focus on extreme eigenvectors}

\author{Tingzhou Yu  \thanks{University of Waterloo, University of Alberta, \textsf{tingzhou.yu@ualberta.com}}}

\maketitle
\abstract{We explore Langevin dynamics in the spherical Sherrington-Kirkpatrick model, delving into the asymptotic energy limit. Our approach involves integro-differential equations, incorporating the Crisanti-Horner-Sommers-Cugliandolo-Kurchan equation from spin glass literature, to analyze the system's size and its temperature-dependent phase transition. Additionally, we conduct an average case complexity analysis, establishing hitting time bounds for the bottom eigenvector of a Wigner matrix. Our investigation also includes the power iteration algorithm, examining its average case complexity in identifying the top eigenvector overlap, with comprehensive complexity bounds.}

\section{Introduction}\label{sec:intro}

Spin glasses are disordered magnetic systems known for their complex behavior, making them a topic of interest in physics \cite{edwards1975theory}. These systems have extensive applications across various fields, including statistics, computer science, and beyond (see e.g., \cite{chatterjee2007estimation, donoho2009message, zdeborova2016statistical}). Among the models studying spin glasses, the Sherrington-Kirkpatrick (SK) model \cite{sherrington1975solvable} is particularly noteworthy for its extensive analysis. This paper primarily focuses on the spherical Sherrington-Kirkpatrick (SSK) model, which serves as the continuous counterpart of the SK model. A central challenge in the study of spin glasses is understanding the equilibrium phase transition (see e.g., \cite{georgii2011gibbs, georgii2001random}), a critical change in the system's properties occurring as temperature decreases. This transition reveals insights about the nature of the system's ground state.

Consider $\mathbb{S}^{N-1}(\sqrt{N})$, the $N-1$ dimensional sphere of radius $\sqrt{N}$ in $\R^N$: $\mathbb{S}^{N-1}(\sqrt{N}):=\{\bm{X}\in \R^N: \|\bm{X}\|^2=N\}$, where $\|\cdot\|_2$ represents the $\ell^2$ norm. The $2$-spin spherical Sherrington-Kirkpatrick (SSK) model is defined by the Hamiltonian:
\begin{equation}
H_{\bJ}(X):=\sum_{1\le i, j \le N} J_{ij}X_iX_j=\bm{X}^T\bJ\bm{X},
\end{equation}
where $\bm{X}=(X_1,\dots, X_N)\in \R^N$ are spin variables on the sphere $\mathbb{S}^{N-1}(\sqrt{N})$ and $\bJ={J_{ij}}_{1\le i, j\le N}$ is the normalized Wigner matrix as defined in \cref{def_wigner}.

However, in some cases, the system may relax to equilibrium so slowly that it never actually reaches it. This phenomenon is known as `aging', and it is integral to the study of spin glass dynamics, both experimentally and theoretically. Aging is a phenomenon that affects the system's decorrelation properties over time. According to \cite{ben2003aging}, this means that the longer the system exists, the longer it takes to forget its past. This phenomenon has been studied extensively by various authors (see e.g., \cite{cugliandolo1993analytical, cugliandolo1994out, vincent2007slow}).
The study of aging and its effect on spin glass dynamics is an active area of physics research, with important implications for the understanding of the low-temperature behavior of these systems.

The foundational mathematical literature on aging in spin glasses was first introduced in \cite{arous2001aging}. In this work, the authors concentrated on such systems, particularly investigating the aging phenomenon in spherical spin glasses. This was aimed at characterizing the low-temperature behaviors of Langevin dynamics in matrix models. A key focus was the study of Langevin dynamics within the Sherrington-Kirkpatrick (SK) model, utilizing correlation functions. These functions adhere to a complex system of integro-differential equations, known as the Crisanti-Horner-Sommers-Cugliandolo-Kurchan (CHSCK) equations \cite{crisanti1992spherical, cugliandolo1993analytical}. Recent advancements in this field include the work of \cite{dembo2020dynamics}, who derived the CHSCK equations for spherical mixed $p$-spin disordered mean-field models. Another significant contribution was made by \cite{dembo2021diffusions}, who employed a general combinatorial method and stochastic Taylor expansion to establish universality in Langevin dynamical systems. Additionally, \cite{liang2022high} explored the signal recovery problem in spiked matrix models using Langevin dynamics.

In this paper, we aim to analyze the asymptotic limit of the energy and focus on the long-time behavior of the system's energy in \cref{sec:long_time}. The energy, described by the Hamiltonian function, is critical in determining the equilibrium properties and governing the dynamics of spherical spin glass models (see e.g., \cite{stein2013spin, auffinger2018energy}). By studying the long-time behavior of the energy function, we can gain insight into the system's energy evolution over time and its interaction with the spin variable dynamics. Notably, we derived an integro-differential equation of energy involving the empirical correlation function in \cref{thm:ntoin_2}. We also established an explicit formula for the energy's limiting behavior in \cref{thm: limit_energy}, where a phase transition was observed. Consequently, we obtained the limiting ratio of the energy to the empirical correlation function of the system in \cref{coro_hk}. The techniques employed in deriving our main results partially rely on the work of \cite{arous2001aging, liang2022high}.

As a potential application, the insights gained from our analysis of the asymptotic limit of the energy can be utilized to develop more efficient algorithms for solving challenging linear algebra problems, such as computing the eigenvectors of Wigner random matrices. However, this hinges on our understanding of the complexity of iterative methods, which have received less attention in the realm of complexity theory \cite{smale1997complexity}.

The complexity of algorithms in linear algebra has been a topic of interest for many years in \cite{smale1997complexity}. While direct algorithms that solve problems in a finite number of steps have been extensively studied, iterative methods such as those required for the matrix eigenvalue problem have received less attention in complexity theory in \cite{smale1997complexity}. The power method is a popular iterative algorithm that approximates the eigenvector corresponding to the dominant eigenvalue. However, for Hermitian random matrices, the complexity of the power method for obtaining a dominant eigenvector is infinite by \cite{kostlan1988complexity}. \cite{kostlan1988complexity} showed that the upper bound of the complexity is $O(N^2\log N)$, conditioned on all the eigenvalues being positive. The upper bound of this complexity is established as $O(N^2\log N)$ under the assumption of all positive eigenvalues. In a separate study, \cite{kostlan1991statistical} explored another algorithm for calculating the dominant vector, revealing that under certain conditions, the average number of iterations needed is $O(\log N+\log |\log \ep|)$. Furthermore, \cite{deift2017universality} studied the efficiency of three algorithms for computing the eigenvalues of sample covariance matrices, concluding that the complexity approximates $O\left(\left(\frac{\log \ep^{-2}}{\log N}-\frac{3}{2}\right)N^{2/3}\log N\right)$, independent of the specific distribution of the matrix entries.

In our work, as presented in \cref{sec:hit_time_gd}, we delve into the complexity of an algorithm employing spherical gradient descent within the aging framework to analyze the equilibrium of the spherical SK model under zero-temperature dynamics (i.e., setting $\beta=\infty$ in the Langevin dynamics defined in \eqref{sde1}). Spherical gradient descent, functioning as an optimization method, updates spin variables in the direction of the negative gradient of the energy function on a unit sphere, effectively serving as a continuous analog of the power method. We focus on the hitting time, the moment when there is a positive overlap between the algorithm's output and the bottom eigenvector, with overlap being a measure of similarity between two spin configurations. Our research contributes to the field by providing bounds for the complexity of computing eigenvectors of Wigner random matrices with finite fourth moments; specifically, we establish a lower bound of $O(N^{2/3})$ and an upper bound of $O(N^{2/3}\log N)$ in \cref{thm:lower_bound_gd}. Similarly, we analyzed the power iteration algorithm, particularly its complexity in reaching the first instance when the algorithm's output and the top eigenvector overlap. Here too, we found the complexity's upper bound to be $O(N^{2/3}\log N)$ and the lower bound to be $O(N^{2/3})$, as detailed in \cref{thm:hit_power}."

\paragraph{Notation:}

For two real-valued functions $f(x), g(x)$, we write  $f(x) = \mathcal{O}(g(x))$ as $x\to \infty$, which means there are constants $C>0$ and $x_0$ such that $|f(x)|\le C |g(x)|$ for all $x\ge x_0$. 

For $f(x)=o(g(x))$ as $x\to \infty$, it means for any $\ep>0$, there exists an $x_0$ such that $|f(x)|\le \ep |g(x)|$ for all $x>x_0$.

Asymptotic Equality $f(x)\si_{x\to\infty} g(x)$: Implies that $\lim_{x\to \infty}f(x)/g(x)=1$, showing that $f(x)$ and $g(x)$ are asymptotically equal as $x$ gets very large.

The definition of the (normalized) Winger matrix we consider in this paper is as follows.
\begin{defn}\label{def_wigner}
Let $N\ge 1$ be an integer. Consider a symmetric (Hermitian) $N\times N$ matrix $\bJ=\{J_{ij}\}_{1\le i,j\le N}=\{\frac{1}{\sqrt{N}}Z_{ij}\}_{1\le i,j\le N}$. Assume that the following conditions hold:
\begin{itemize}
    \item The upper-triangular entries $\{Z_{ij}\}_{1\le i\le j \le N}$ are  independent real (complex) random variables with mean zero;
    \item The diagonal entries $\{Z_{ii}\}_{1\le i \le N}$ are identically distributed with finite variance, and the off-diagonal entries $\{Z_{ij}\}_{1\le i < j\le N}$ are identically distributed with unit variance;
    \item Furthermore, in the context of the Hermitian case, assume that $\E[(Z_{12})^2]=0$.
\end{itemize}
The matrix $\bJ$ is called the \textbf{normalized symmetric (Hermitian) Wigner matrix}.
\end{defn}

\begin{remark}
    For cases where the random variables $Z_{ij}$ and $Z_{ii}$ are real Gaussian and $\E[|Z_{ii}|^2]=2$ for $1\le i, j \le N$, the Wigner matrix $\bJ$ is called the Gaussian Orthogonal Ensemble (GOE). Likewise, when the entries $Z_{ij}$ are complex Gaussian and $Z_{ii}$ are real Gaussian with $\E[|Z_{ii}|^2]=1$, the Wigner matrix $\bJ$ is called the Gaussian Unitary Ensemble (GUE).
\end{remark}

%%%%%%%%%%%%%%%%%%%%%%%%%%%%%%%%%%%%%%%%%%%%%%%%%%%%%%%%%%%%
\section{Asymptotic energy dynamics in the spherical SK model
}\label{sec:long_time}

In this section, we explore the asymptotic energy dynamics of the spherical Sherrington-Kirkpatrick model. Our focus is on the Langevin dynamics, through which we analyze the system's behavior over time. We will present our main results, including the derivation of integro-differential equations and the examination of long-term energy behavior and phase transitions in the model.

\subsection{Main results}

We consider the Langevin dynamics for the Sherrington-Kirkpatrick (SK) model defined by the following system of stochastic differential equations (SDEs) as in \cite{arous2001aging}:
\begin{equation}\label{sde1}
    dX_t^i=\sum_{j=1}^N J_{ij}X_t^jdt-f'\lr{\rc{N}\sumjN (X_t^j)^2}X_t^jdt+\be^{-1/2}dW_t^i,
\end{equation}
where $\bJ=\{J_{ij}\}_{1\le i,j\le N}$ is a symmetric matrix of centered Gaussian random variables such that $\E[J_{ij}^2]=\rc{N}$ and $\E[J_{ii}^2]=\frac{2}{N}$ for $1\le i<j \le N$, $f: [0,\iy)\to \R$ satisfies $f'$ to be non-negative and Lipschitz, $\be$ is a positive constant, and $\{W^i_t\}_{1\le i\le N}$ is an $N-$dimensional Brownian motion, independent of $\{J_{ij}\}_{1\le i, j \le N}$ and of the initial data $\{X_0^i\}_{1\le i\le N}$. 

For any $N\ge 1$ and $T\ge 0$, the SDE \eqref{sde1} has a unique strong solution $\bm{X}_t=\{X_t^i: 1\le i\le N, t\in [0,T]\}$ on $\cC([0,T], \R^N)$. See \cite[Lemma 6.7]{arous2001aging} for a proof.

The second term in \eqref{sde1} is a Lagrange multiplier in order to implement a smooth spherical constraint \cite{arous2001aging}. The simplification caused by the SSK model is the invariance under rotation for the SDE \eqref{sde1}. 

We write $\bJ=G^TDG$, where $G$ is an orthogonal matrix with the uniform law on the sphere and $D=\text{diag}(\sigma^1,\dots, \sigma^N)$ is the diagonal matrix of the eigenvalues $\{\sigma^i\}_{1\le i\le N}$ of $\bJ$. As $N\to \iy$, we have $\rc{N}\sumiN \delta_{\sigma^i}$ converges weakly to the semicircle law $\mu_D$ with compact support $[-2,2]$ by \cite[Theorem 2.4.2]{tao2012random}.  Recall that the density function of semicircle law $\mu_D$ is given by
\begin{equation}\label{eq: semi_circle}
   d \mu_D=\frac{1}{2\pi}\sqrt{4-x^2}\1_{\{-2\le x\le 2\}}d\,x.
\end{equation}

To simplify the SDE \eqref{sde1}, we let both sides of \eqref{sde1} be multiplied by the rotation matrix $G$ which is invariant under rotation. We take $\bm{Y}_t:=G\bm{X}_t$ and $B_t:=G W_t$. Then the SDE under the rotation is given by
\begin{equation}\label{eq:sdey}
    dY_t^i=\lr{\sigma^i-f'\lr{\|Y_t\|_2^2/N}}Y_t^idt+\be^{-1/2}dB_t^i,
\end{equation}
where $\|\cdot\|_2$ is the $\ell^2$ norm.

Denote by \begin{equation}\label{eq: def_kn}
    K_N(t,s):=\rc{N}\sumiN X_t^iX_s^i
\end{equation}  the empirical correlation function. We use abbreviated notation $K_N(t):=K_N(t,t)$ for convenience.  

Ben Arous, Dembo, and Guionnet studied the dynamics of the empirical correlation $K_N$ and the limiting point as $N\to \infty$ ($N$ is the size of the system) in \cite{arous2001aging}, which is the unique solution to a CHSCK equation as follows.

\begin{thm}\cite[Theorem 2.3]{arous2001aging}\label{thm:ntoin}
Assume that the initial data $\{X_0^i\}_{1\le i\le N}$ are i.i.d with law $\mu_0$ so that $\E_{X\sim \mu_0} [e^{\al X}]<\iy$ for some $\al>0$. Fix $T\ge 0$. As $N\to \iy$, $K_N$ converges almost surely to deterministic limits $K$. Recall that $\mu_D$ is the semicircle law. Moreover, the limit $K$ is the unique solution to the following integro-differential equation:
\begin{align*}
  K(t,s)&=e^{-\int_0^tf'(K(w))dw-\int_0^sf'(K(w))dw}\underset{(\sigma, X_0)\sim \pi^{\infty}}{\E}[e^{\si(t+s)}X_0^2]\\
  &+\be^{-1}\int_0^{t\land s}e^{-\int_r^tf'(K(w))dw-\int_r^sf'(K(w))dw}\underset{(\sigma, X_0)\sim \pi^{\infty}}{\E}[e^{\si(t+s-2r)}]dr,
\end{align*}
where $\pi^{\iy}=\mu_D\otimes \mu_0$  and here we write $K(s):=K(s.s)$.
\end{thm}

\begin{remark}
As emphasized in \cite{arous2001aging}, the aging is very dependent on initial conditions. In addition to considering i.i.d. initial condition, the author also considers other three types of initial conditions: the rotated independent initial conditions, the top eigenvector initial conditions, and the stationary initial conditions.
\end{remark}

\begin{remark}
Based on the thermodynamic limit of $K_N(t,s)$ as $N\to \iy$, the authors study the long time evaluations of $K(t,s)$ and established a dynamical phase transition in terms of the asymptotic of $K(t,s)$ in \cite[Proposition 3.2]{arous2001aging}. This is a first mathematical proof of the aging phenomenon.
\end{remark}

Next, we similarly consider how to describe how the energy of the system evolves over time. Recall that the quadratic Hamiltonian of SSK model is defined by $H_{\bJ}(\bm{X}_t)=\bm{X}_t^T\bJ \bm{X}_t$. Note that we have $H_{\bJ}(\bm{Y}_t)=\bm{Y}_t^T D \bm{Y}_t$,  where $\bm{Y}_t=GX_t$ and $\bJ=G^TDG$.

Let \begin{equation}\label{eq:energy}
H_N(t):=\rc{N} H_{\bJ}(\bm{Y}_t)=\rc{N}\sumiN\si^i(Y_t^i)^2
\end{equation} 
be the \textbf{energy} of the system.

Our first result characterizes the limiting behavior of the energy  $H_N(t)$ of the SSK model as $N\to \iy$ for $t\in [0,T]$ as follows.

\begin{thm}\label{thm:ntoin_2}
Assume that the initial data $\{X_0^i\}_{1\le i\le N}$ are i.i.d with law $\mu_0$ so that $\E_{X\sim \mu_0} [e^{\al X}]<\iy$ for some $\al>0$. Fix $T\ge 0$. Let $K$ be the solution defined as in Theorem \ref{thm:ntoin}. As $N\to \iy$, $H_N$ converges almost surely to deterministic limits $H$. Moreover, the limit $H$ is the unique solution to the following integro-differential equation:
\begin{align}
    H(t)&=e^{-2\int_0^tf'(K(w))dw}\underset{(\sigma, X_0)\sim \pi^{\infty}}{\E}[\si e^{2\si t}X_0^2]  \label{ex1}\\
    &+\be^{-1}\int_0^t e^{-2\int_s^t f'(K(w))dw}\underset{(\sigma, X_0)\sim \pi^{\infty}}{\E}[\si e^{2\si (t-s)}]ds \label{ex2}
\end{align}
where $\pi^{\iy}=\mu_D\otimes \mu_0$  and here we write $K(s):=K(s.s)$.
\end{thm}

Theorem \ref{thm:ntoin_2} provides a precise characterization of $H(t)$. However, the expression of $H(t)$ is unclear because it involves the fixed point equation of $K(t)$ in Theorem \ref{thm:ntoin}. The key point of this model is that we can exactly study the long time behavior of the energy $H(t)$ as $t\to \iy$.

In order to precisely determine the limit of the energy, we will define $K(0,0)$ to be 1 and take function \begin{equation}\label{eq:func_sys}
f(x)=\frac{cx^2}{2},\end{equation}
where $c$ is a positive constant.

Let $m: \R\setminus \mbox{supp}(\mu_D) \to \R$ be the Stieljes transform of the probability measure $\mu_D$ given by
\begin{equation}\label{eq:stj_tran}
    m(s)=\E_{\si\sim \mu_D}\left[\rc{s-\si}\right]=\frac{2}{s+\sq{s^2-4}}.
\end{equation}

Let $\beta_c$ be the critical temperature such that
\begin{equation}\label{eq:cri_tem}
    \beta_c=\fc{c}{4}m(2)=\fc{c}{4}.
\end{equation}

Our second result is as follows.

\begin{thm}\label{thm: limit_energy}
Assume that $K(0,0)=1$ and $f(x)=\frac{cx^2}{2}$ for some positive constants $c>0$. Let $H$ be the unique solution given in Theorem \ref{thm:ntoin}. Then, for $\beta \leq \beta_c$, we have
\begin{equation}
\lim_{t \to \infty} H(t) = 0.
\end{equation}
For $\beta > \beta_c$, we have
\begin{equation}
\lim_{t \to \infty} H(t) = \frac{4\beta-c}{2^{\frac{5}{2}}\sqrt{\pi}c\beta}+\frac{1}{2}\beta^{-1}.
\end{equation}

\end{thm}

Theorem \ref{thm: limit_energy} describes the limit of energy function $H(t)$ as $t\to \iy$. This concludes a dynamical phase transition phenomenon. See Figure \ref{fig:plot of h} for the existence of a jump discontinuity in the asymptotic limit of the function $H(t)$, where we set $c=1.$

\begin{figure}
    \centering
    \includegraphics[width=0.4\textwidth]{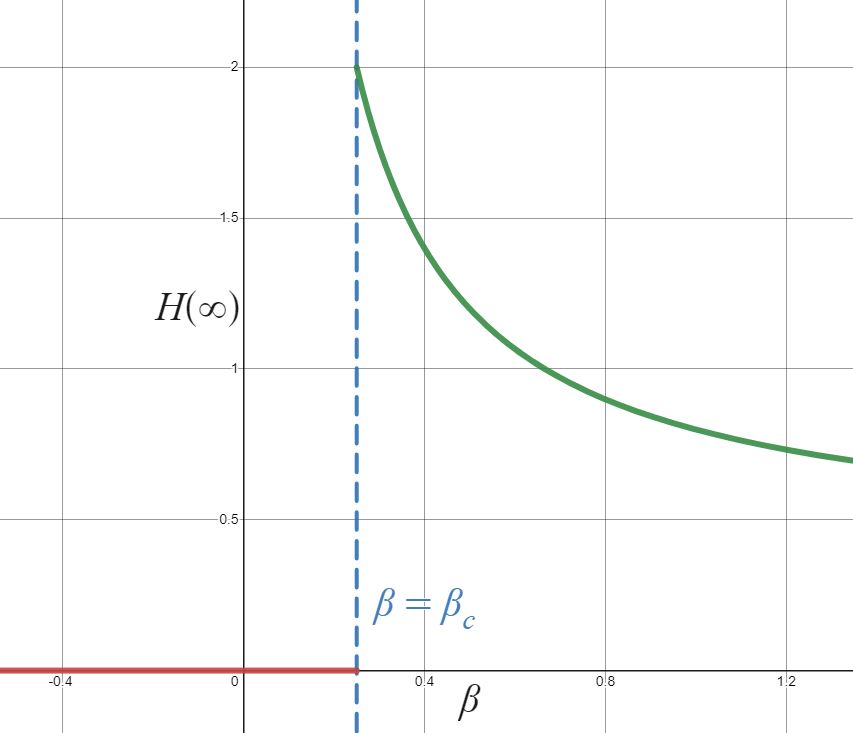}
    \caption{In this figure, we set $c=1$ and plot the limiting behavior of the function $H(t)$ as $t$ approaches infinity. $H(t)$ exhibits a jump discontinuity in the phase transition at the critical inverse temperature $\beta_c=0.25$.}
    \label{fig:plot of h}
\end{figure}

The proof of Theorem \ref{thm: limit_energy}  utilizes tools and techniques from the paper \cite{arous2001aging}, with some modifications made to their results. We borrow the notation from \cite{arous2001aging} and write \begin{equation}\label{eq: eq_r_t}R(t):=\exp\lr{2\int_0^t f'\lr{K(w)}dw}\end{equation} with $K(w)=K(w,w)$. 

Then the expression of $H(t)$ in Theorem \ref{thm:ntoin_2} becomes 
\begin{equation}\label{eq:ht}
   H(t)= R(t)^{-1}\lr{\E[\si e^{2\si t}]+\beta^{-1}\int_0^t R(r) \E[\si e^{2\si (t-r)}] dr}
\end{equation}
Note that the limit of $H(t)$ is governed by the asymptotic of the derivative of the moment generating function of $\si$. So it suffices to characteristic the limit of $R(t)$ and $\EX{\si e^{2\si t}}$.

Similarly, we consider the asymptotic limit of $H(t)/K(t)$ as $t\to \iy$.
\begin{corollary}\label{coro_hk}
Assume the same setting as in Theorem \ref{thm: limit_energy}. Then for $\beta \leq \beta_c$, we have
\begin{equation}
\lim_{t \to \infty} \frac{H(t)}{K(t)} = 0.
\end{equation}
For $\beta > \beta_c$, we have
\begin{equation}
\lim_{t \to \infty} \frac{H(t)}{K(t)} = \frac{2^{-3/2}(4\be-c)+\sq{\pi}c}{2^{-5/2}(4\be-c)+\sq{\pi}c}
\end{equation}
\end{corollary}

\begin{figure}
    \centering
    \includegraphics[width=0.4\textwidth]{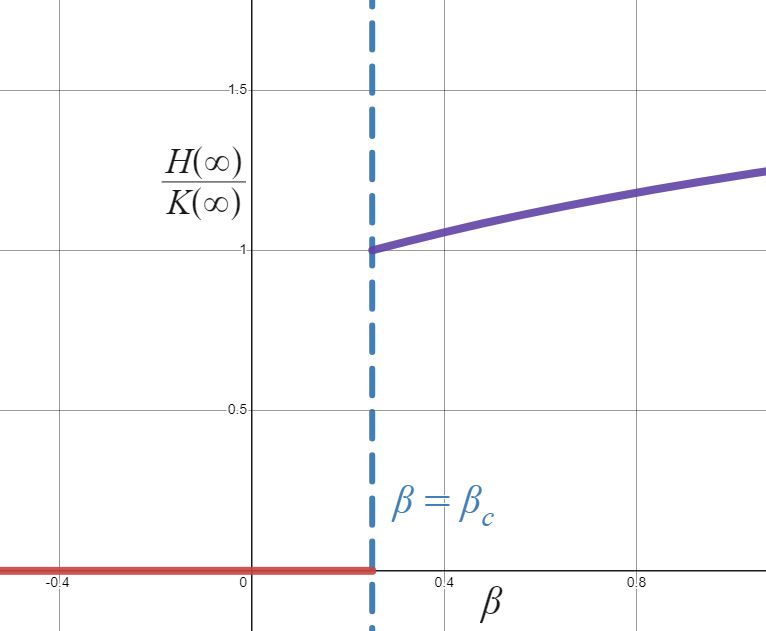}
    \caption{In this figure, we set $c=1$ and plot the limiting behavior of the function $H(t)/K(t)$ as $t$ approaches infinity. There is a jump discontinuity in the phase transition at the critical inverse temperature $\beta_c=0.25$.}
    \label{fig:plot of hk}
\end{figure}

\subsection{Proof of Theorem \ref{thm: limit_energy}}

In proof of our main results, we will need to use the Bessel function. Let us first revisit its basic definition and some relevant results that we will be using.

An alternative definition of the (modified) Bessel function, for integer values of $n$, is possible using integral representation:
\begin{defn}\cite[Section 9.1]{abramowitz1948handbook}\label{def:bessel}
The Bessel function is given by
\begin{equation}
    B_n(x):=\frac{i^{-n}}{\pi}\int_0^\pi e^{ix\cos\theta}\cos(n\theta)d\theta.
\end{equation}
for $n\in\N$ and $x\in \R$.

The modified Bessel function is given by
\begin{equation}
    I_n(x):=\frac{1}{\pi}\int_0^\pi e^{x\cos\theta}\cos(n\theta)d\theta.
\end{equation}
for $n\in\N$ and $x\in \R$.
\end{defn}

The relation between the Bessel function and the modified Bessel function is given by
\begin{equation}\label{eq:rela_m_bess}
    I_n(x)=e^{-in\pi/2}B_n(xe^{i\pi/2})
\end{equation}
as shown in \cite[Section 9.6]{abramowitz1948handbook}.

We will use the following lemma about the recurrence relation and derivatives of modified Bessel functions.

\begin{lemma}\cite[Section 9.6]{abramowitz1948handbook}\label{lem:recu_mbess}
    Let $I_n$ be the modified Bessel function defined as in Definition \ref{def:bessel}. For $n\in \N$,
    \begin{equation}
        I_n'(x)=I_{n+1}(x)+\frac{n}{x}I_n(x),
    \end{equation}
    and $I'_0(x)=I_1(x)$.
\end{lemma}

By \cite[Section 9.6]{abramowitz1948handbook}, we have the following asymptotic results of modified Bessel functions.
\begin{lemma}\label{lem:asy_mbess}
     Let $I_n$ be the modified Bessel function defined as in Definition \ref{def:bessel}. For $n\in \N$, as $x\to \infty$ we have
     \begin{equation}
\lim_{x\to \infty} \frac{I_n(x)}{x^{-1/2}e^x}=\frac{1}{\sq{2\pi}}.
     \end{equation}
\end{lemma}

We will give the representation of the characteristic function of the semicircle law by the Bessel function.

\begin{lemma}\label{lem:char_semi}
Let $B_1(t)$ be the Bessel function defined as in Definition \ref{def:bessel} for $n=1$. Recall that the eigenvalues $\sigma$ of $N\times N$ normalized symmetric Wigner matrix $\bJ$ follow the semicircle law with distribution $\mu_D$ as in \eqref{eq: semi_circle}. Then we have
\begin{equation}
    \E[e^{i t\si}]=\frac{B_1(2t)}{t}
\end{equation}
\end{lemma}

\begin{proof}
By \cite[Theorem 6.2.3]{chung2001course}, it is enough to calculate the inverse of the characteristic function of $ B_1(2t)/t$ is the density of the semicircle law.

Note that by the inversion formula we have
\begin{align*} 
\dfrac{1}{2\pi} \int_{\mathbb{R}} \dfrac{1}{t} B_1\left(2t\right) e^{-itx} dt &= \dfrac{1}{2i\pi^2} \int_{\mathbb{R}}\int_{0}^\pi \dfrac{1}{t} e^{2it\cos \theta }\cos \theta e^{-itx} d\theta dt  \\[7pt] 
&=  \dfrac{1}{2i\pi^2} \int_{0}^\pi  \cos \theta  \underbrace{ \int_{\mathbb{R}}  \dfrac{1}{t}  e^{it(2\cos \theta-x)} dt }_{=-i\pi\cdot \text{Sign}(2\cos\theta-x) } d\theta \\[7pt]
&= - \dfrac{1}{2\pi} \int_{0}^\pi \cos \theta \cdot \text{sign}(2\cos \theta-x)  d\theta,
\end{align*}
where $\text{Sign}(\cdot)$ is the Sign function.

Clearly, we have $ \text{Sign} (2 \cos \theta-x)  = -1$ for $x>2$, and $ \text{sgn} (2 \cos \theta-x)  = 1$ for $x<-2$. For both cases, the above integral is zero due to $ \int_{0}^\pi \cos \theta d\theta  = 0$. 

Consider the case that $ -2\le x\le 2 $. Set $ u=2\cos \theta-x  $. The above integral becomes
\begin{align*}
-\dfrac{1}{2\pi}\int_{-2-x}^{2-x} \dfrac{u+x}{2} \dfrac{1}{\sqrt{1-\left(\frac{u+x}{2}\right)^2}}  \text{Sign}(u)  du
&= \dfrac{1}{2\pi}\Bigg\{\int_{-2-x}^{0} \dfrac{u+x}{2} \dfrac{1}{\sqrt{1-\left(\frac{u+x}{2}\right)^2}}  du-\int_{0}^{2-x} \dfrac{u+x}{2} \dfrac{1}{\sqrt{1-\left(\frac{u+x}{2}\right)^2}}   du\Bigg\}\\[7pt]
&= \dfrac{1}{2\pi}  \left(    2 \int_{-1}^{\frac{x}{2}}  \dfrac{y}{\sqrt{1-y^2}}  dy - 2\int_{\frac{x}{2}}^{1}  \dfrac{y}{\sqrt{1-y^2}}   dy   \right)\\[7pt]
&= \dfrac{1}{2\pi} \sqrt{4-x^2},
\end{align*}
where we take variable $y=(u+x)/2$ in the second line for $-2\le x\le 2$.
\end{proof}

The following Lemma derives the asymptotic limit of the derivative of the moment generating function of $\si$.
\begin{lemma}\label{lem:asy_mgf}
Recall that the eigenvalues $\sigma$ of $N\times N$ normalized symmetric Wigner matrix $\bJ$ follow the semicircle law with distribution $\mu_D$ as in \eqref{eq: semi_circle}. Then we have:
\begin{equation}
\lim_{t \to \infty} \frac{\E_{\si\sim \mu_D}[\sigma e^{t\sigma}]}{t^{-3/2}e^{2 t}}= \frac{1}{2\sqrt{\pi}}.
\end{equation}
\end{lemma}

%\textcolor{red}{[note: need to modified the coefficient $\frac{1}{2\sq{\pi}}$.]}

\begin{proof}
Substitute $t$ by $it$ in Lemma \ref{lem:char_semi}, then we have 
\begin{equation}
    \E[e^{t\si}]=\frac{B_1(-2i t)}{-2it}
\end{equation}
where $B_1(\cdot)$ is the Bessel function.

By equation \eqref{eq:rela_m_bess}, we get
\begin{equation}
    \E[e^{t\si}]=\frac{I_1(2t)}{2t},
\end{equation}
where $I_1(2t)$ is the modified Bessel function. 

Thus, the derivative of the moment generating function can be expressed as
\begin{equation}\label{eq:expres_mgf1}
     \E[\si e^{t\si}]=-t^{-2}I_1(2t)+2t^{-1}I_1'(2t)
\end{equation}

Combine \eqref{eq:expres_mgf1} and Lemma \ref{lem:recu_mbess} we have 
\begin{equation}\label{eq:expres_mgf}
     \E[\si e^{t\si}]=\fc{I_2(2t)}{t}.
\end{equation}

By Lemma \ref{lem:asy_mbess}, it yields the desired result.

\end{proof}

Similarly, we have the following result.
\begin{lemma}\label{asy_mgf_1}
    Assume the same setting holds as in Lemma \ref{lem:asy_mgf}, we have
    \begin{equation}
        \lim_{t \to \infty} \frac{\E_{\si \sim \mu_D}[ e^{t\sigma}]}{t^{-3/2}e^{2 t}}= \frac{1}{4\sqrt{\pi}}.
    \end{equation}
\end{lemma}

Let $m(s)$ be defined as in \eqref{eq:stj_tran}. We define
\begin{equation}
p(s, \beta):=\frac{2\beta s}{c}-m(s).
\end{equation}

Recall that $\beta_c$ is the critical temperature defined in \eqref{eq:cri_tem}. For any $\beta\in (0, \beta_c)$, there exists a unique solution of $p(s,\beta)=0$ on the interval $(2,\infty)$, denoted by $s_\be$. We can solve for $s_\be$ as $s_\be=2(1-(1-\be/\be_c)^2)^{-1/2}$. For $\beta>\beta_c$, we simply define $s_\be=2$.

We can get the asymptotic limit of $R(t)$ as follows.
\begin{lemma}\cite[Lemma 3.3]{arous2001aging}\label{lem:asyofR}
Recall that $s_\be=2(1-(1-\be/\be_c)^2)^{-1/2}$ for any $\beta\in (0, \beta_c)$, and $s_\be=2$ for $\beta\ge \beta_c$. Let $\Psi=0$ for $\be<\be_c$, $\Psi=\frac{1}{2}$ for $\be=\be_c$, and $\Psi=\frac{3}{2}$ for $\be>\be_c$. Then there exists a constant $C_\be>0$ such that
\begin{equation}
\lim_{x\to \infty}\frac{R(x)}{x^{-\Psi}e^{2xs_{\be}}}=C_{\be}.
\end{equation}
Moreover, we have
\[
C_\be=\begin{cases}
\fc{\be (cm(s_\be)+1)}{2\be-cm'(s_\be)},\quad &\be<\be_c \\
\fc{\be(c+1)}{c},\quad &\be=\be_c\\
\fc{c\be (4\be+1)}{(4\be -c)^2}, \quad & \be>\be_c
\end{cases} 
    \]
where $c$ is the coefficient constant defined in \eqref{eq:func_sys} and $m(s)$ is defined as in \eqref{eq:stj_tran}.
\end{lemma}

Combining Lemma \ref{lem:asy_mgf} and Lemma \ref{lem:asyofR} we can characterize the limit of $H(t)$ as $t\to \iy$. In order to give a precise result of the asymptotic limit, we also need to do Laplace transformation on both sides of the equation \eqref{eq:ht} to get an identity as follows.

Define the Laplace transform of the function $R(t)$ for $z>2$ by
\begin{equation}
    \Ls_R(z):=\int_0^{\iy} e^{-2zt}R(t)dt.
\end{equation}

\begin{lemma}\label{lem:lap}
The Laplace transform $\Ls_R(z)$ satisfies the equation
\begin{equation}
    2z\Ls_R(z)-1=cm(z)(1+\be^{-1}\Ls_R(z)),
\end{equation}
where $c$ is the coefficient constant defined in \eqref{eq:func_sys} and $m(s)$ is defined as in \eqref{eq:stj_tran}.
\end{lemma}

\begin{proof}
Note that 
\[K(t)R(t)=K(t)e^{2c\int_0^t K(w)dw}=\frac{1}{2c}\partial_t R(t).
\]
Then we have the linear Volterra integro-differential equation
\begin{equation}\label{eq:vide}
    R'(t)=2cK(t)R(t)=2c\lr{\E[e^{2\si t}]+\be^{-1}\int_0^t R(r)\E[e^{2\si(t-r)}]dr}.
\end{equation}

The Laplace transform of the LHS in \eqref{eq:vide} is 
\[\Ls_{R'}(z)=-R(0)+2z \Ls_R(z)=-1+2z \Ls_R(z).
\]

Note that the term inside the integral on RHS in \eqref{eq:vide} can be expressed as the convolution of $R(t)$ and $e^{2\si t}$. We write it as $(e^{2\si \cdot}*R)(t)$ and use the fact that the Laplace transform of this one is equal to product of the Laplace transform of each function.

Thus, the RHS becomes
\begin{align*}
    \int_0^{\iy} e^{-2zt}R'(t)dt&=\int_0^{\iy} e^{-2zt}\lr{2c\lr{\E[e^{2\si t}]+\be^{-1}\int_0^t R(r)\E[e^{2\si(t-r)}]dr}}dt\\
    &=c\EX{\rc{z-\si}}+2c\be^{-1}\EX{\int_0^\iy e^{-2t(z-\si)}dt}\Ls_g(z)\\
    &=c\EX{\rc{z-\si}}+c\be^{-1}\EX{\rc{z-\si}}\Ls_g(z)
\end{align*}

Hence, combining the Laplace transforms of the left and right sides, we obtain
\begin{align*}
    2z\Ls_R(z)=1+cm(z)+c\be^{-1}m(z)\Ls_R(z).
\end{align*}

\end{proof}

We now turn to the proof of Theorem \ref{thm: limit_energy}.

\begin{proof}[Proof of Theorem \ref{thm: limit_energy}]
\begin{enumerate}[(i)]
    \item We start by considering the case where $\beta > \beta_c$.

Using Lemma \ref{lem:asyofR}, we obtain an asymptotic limit for $R(t)$ as:
$$R(t) \sim_{t\to\infty} C_\beta t^{-3/2}e^{4t},$$
where $C_\beta = \frac{c\beta (4\beta+1)}{(4\beta - c)^2}$.

Combining the asymptotic limit for $R(t)$ and Lemma \ref{lem:asy_mgf}, we notice that the limit of the first term of $H(t)$ defined as in \eqref{eq:ht} is:
\begin{equation}\lim_{t\to\infty} R(t)^{-1}\E[\sigma e^{2\sigma t}] =\lim_{t\to\iy}\frac{\fc{\E[\si e^{2\si t}]}{(2t)^{-3/2}e^{4t}}(2t)^{-3/2}e^{4t}}{\fc{R(t)}{C_\be t^{-3/2}e^{4t}}C_\be t^{-3/2}e^{4t}}=\frac{2^{-5/2}}{\sq{\pi}C_\beta}.\end{equation}

Next, we multiply the integral in equation \eqref{eq:ht} by the asymptotic limit of $R(t)$ and split it into three parts:
for $t\to \infty, x\to \infty, $ and $x/t\to 0$,
\begin{align*}
    C_\be^{-1}t^{\fc{3}{2}}e^{-4t}\int_0^t R(r) \E[\si e^{2\si (t-r)}] dr
    &=C_\be^{-1}t^{\fc{3}{2}}e^{-4t}\Bigg(\int_0^x R(r) \E[\si e^{2\si (t-r)}] dr\\
    &+\int_x^{t-x} R(r) \E[\si e^{2\si (t-r)}] dr+\int_{t-x}^t R(r) \E[\si e^{2\si (t-r)}] dr\Bigg)\\
    &=C_\be^{-1}t^{\fc{3}{2}}e^{-4t}\Bigg(\int_0^x R(r) \E[\si e^{2\si (t-r)}] dr\\
    &+\int_x^{t-x} R(r) \E[\si e^{2\si (t-r)}] dr+\int_0^x R(t-r) \E[\si e^{2\si (r)}] dr\Bigg)\\
    &=:I_1+I_2+I_3.
\end{align*}

We now estimate each term separately. For $I_1$, we have:
%\begin{align*}
%I_1 &= 2^{-\frac{3}{2}}C_\beta^{-1}\int_0^x R(r)e^{-4r}\Big(\frac{t}{t-r}\Big)^{\frac{3}{2}}dr\cdot\Big(\frac{1}{\sqrt{\pi}}+o(1)\Big) \
%&= 2^{-\frac{3}{2}}C_\beta^{-1}\int_0^x R(r)e^{-4r}\cdot\Big(\frac{1}{\sqrt{\pi}}+o(1)\Big)dr.
%\end{align*}
\begin{align*}
    I_1&=C_\be^{-1}t^{\fc{3}{2}}e^{-4t}\int_0^x R(r) \E[\si e^{2\si (t-r)}] dr\\
    &=C_\be^{-1}t^{\fc{3}{2}}e^{-4t}\int_0^x R(r) \fc{\E[\si e^{2\si (t-r)}]}{(2(t-r))^{-\frac{3}{2}}e^{4(t-r)}}(2(t-r))^{-\frac{3}{2}}e^{4(t-r)} dr\\
    &=2^{-\fc{3}{2}}C_\be^{-1}\int_0^x R(r)e^{-4r}\lr{\rc{2\sq{\pi}}+o(1)}\lr{\fc{t}{t-r}}^{\fc{3}{2}}dr\\
    &=2^{-\fc{3}{2}}C_\be^{-1}\int_0^x R(r)e^{-4r}\lr{\rc{2\sq{\pi}}+o(1)}(1+o(1))dr\\
    &=2^{-\fc{5}{2}}C_\be^{-1}\rc{\sq{\pi}}\int_0^x R(r)e^{-4r}dr
\end{align*}

For $I_2$, we can show that it is of smaller order than $I_1$ and $I_3$ and can be neglected. Indeed, we have
\begin{align*}
    I_2&=C_\be^{-1}t^{\fc{3}{2}}e^{-4t}\int_x^{t-x} R(r) \E[\si e^{2\si (t-r)}] dr\\
    &=C_\be^{-1}t^{\fc{3}{2}}e^{-4t}\int_x^{t-x}\fc{R(r)}{C_\be r^{-\fc{3}{2}}e^{4r}}C_\be r^{-\fc{3}{2}}e^{4r}\fc{\E[\si e^{2\si (t-r)}]}{(2(t-r))^{-\frac{3}{2}}e^{4(t-r)}}(2(t-r))^{-\frac{3}{2}}e^{4(t-r)}dr\\
    &=C_\be^{-1}t^{\fc{3}{2}}e^{-4t}\int_x^{t-x}C_\be r^{-\fc{3}{2}}e^{4r}\lr{\rc{2\sq{\pi}}+o(1)}(2(t-r))^{-\frac{3}{2}}e^{4(t-r)}dr\\
    &=2^{-\frac{5}{2}}\rc{\sq{\pi}}\int_x^{t-x} \lr{\fc{t}{r(t-r)}}^{\fc{3}{2}}dr\\
    &=o(1)
\end{align*}

Finally, for $I_3$, we have:
$$I_3 = \int_0^x \E[\si e^{2\si r}]e^{-4r}dr.$$

%Note that the limit of first term of $H(t)$ in \eqref{eq:ht} is
%\begin{equation}
 %   \lim_{t\to\iy} R(t)^{-1}\E[\si e^{2\si t}]=\lim_{t\to\iy}\frac{\fc{\E[\si e^{2\si t}]}{(2t)^{-3/2}e^{4t}}(2t)^{-3/2}e^{4t}}{\fc{R(t)}{C_\be t^{-3/2}e^{4t}}C_\be t^{-3/2}e^{4t}}=2^{-3/2}\pi^{-1/2}C_\be^{-1}.
%\end{equation}

%The only thing left to do is to estimate the second term in \eqref{eq:ht}. Note that the asymptotic limit of $R(t)$ is $C_\be t^{-3/2}e^{4t}$ as $t\to \infty$. So we consider the integral given in \eqref{eq:ht} multiply by the asymptotic limit of $R(t)$: for $t\to \infty, x\to \infty, $ and $x/t\to 0$,
%\begin{align*}
%    C_\be^{-1}t^{\fc{3}{2}}e^{-4t}\int_0^t R(r) \E[\si e^{2\si (t-r)}] dr
%    &=C_\be^{-1}t^{\fc{3}{2}}e^{-4t}\Bigg(\int_0^x R(r) \E[\si e^{2\si (t-r)}] %dr\\
%    &+\int_x^{t-x} R(r) \E[\si e^{2\si (t-r)}] dr+\int_{t-x}^t R(r) \E[\si e^{2\si (t-r)}] dr\Bigg)\\
 %   &=C_\be^{-1}t^{\fc{3}{2}}e^{-4t}\Bigg(\int_0^x R(r) \E[\si e^{2\si (t-r)}] dr\\
%    &+\int_x^{t-x} R(r) \E[\si e^{2\si (t-r)}] dr+\int_0^x R(t-r) \E[\si e^{2\si (r)}] dr\Bigg)\\
%    &=:I_1+I_2+I_3.
%\end{align*}

Note that we have
$$
\int_0^\iy |R(r)|e^{-4r}dr<\iy \mbox{ and }\, \int_0^\iy \E[\si e^{2\si r}]e^{-4r}dr<\iy.
$$

Then we have
\begin{equation}\label{eq:lim_ht_use}
    \lim_{t\to\iy} H(t)=\fc{2^{-\fc{5}{2}}}{C_\be \sq{\pi}}+\fc{2^{-\fc{5}{2}}\be^{-1}}{C_\be \sq{\pi}}\int_0^\iy R(r)e^{-4r}dr+\be^{-1}\int_0^\iy \E[\si e^{2\si r}]e^{-4r}dr.
\end{equation}

By Lemma \ref{lem:lap}, we have
\begin{equation}\label{eq:thm2_1}
\int_0^\iy e^{-4r}R(r)dr=\Ls_R(2)=\fc{\be(1+c)}{4\be-c}
\end{equation}

Also, note that
\begin{align}
    \int_0^\iy \E[\si e^{2\si r}]e^{-4r}dr&=\EX{\si\int_0^\iy e^{-2(2-\si)r}dr}\notag\\
    &=\EX{\fc{\si}{2(2-\si)}}\notag\\
    &=\fc{1}{2}\lr{-1+\EX{\fc{2}{2-\si}}}=\fc{1}{2}\label{eq:proof_res_333}.
\end{align}

Hence, plug \eqref{eq:thm2_1} and \eqref{eq:proof_res_333} into \eqref{eq:lim_ht_use} we have
\begin{equation}
     \lim_{t\to\iy} H(t)=\fc{2^{-\fc{5}{2}}}{C_\be \sq{\pi}}+\lr{\fc{2^{-\fc{5}{2}}}{C_\be \sq{\pi}}}\lr{\fc{1+c}{4\be-c}}+\fc{1}{2}\be^{-1}
     =\fc{2^{-\fc{5}{2}}(4\be-c)}{\sq{\pi}c\be}+\fc{1}{2}\be^{-1}.
\end{equation}

 \item As $\be=\be_c$, by Lemma
 \ref{lem:asyofR},  we have $$R(t)\sim_{t\uparrow \iy} C_\be t^{-1/2}e^{4t}$$
where
$C_\be=\fc{\be(c+1)}{c}$.
 
By Lemma \ref{lem:asy_mgf} and  Lemma  \ref{lem:asyofR}, the first term $R^{-1}(t)\E[\si e^{2\si t}]$ in $H(t)$ converges to $0$ as $t\to \infty$.
 
Note that we have for $x\to \iy, t\to \iy, x/t\to 0$
\begin{align*}
    C_\be^{-1}t^{\fc{1}{2}}e^{-4t}\int_0^t R(r) \E[\si e^{2\si (t-r)}] dr&=
     C_\be^{-1}t^{\fc{1}{2}}e^{-4t}\Bigg(\int_0^x R(r) \E[\si e^{2\si (t-r)}] dr\\
     &+\int_x^{t-x} R(r) \E[\si e^{2\si (t-r)}] dr
     +\int_{0}^x R(t-r) \E[\si e^{2\si r}] dr\Bigg)\\
     &=:E_1+E_2+E_3
\end{align*}
 
Then we have
\begin{align*}
    E_1&=2^{-\fc{3}{2}}C_\be^{-1}\int_0^x R(r)e^{-4r}\fc{t^{\rc{2}}}{(t-r)^{\fc{3}{2}}}\lr{\rc{2\sq{\pi}}+o(1)}dr=o(1)
\end{align*}
where it follows from $\fc{t^{\rc{2}}}{(t-r)^{\fc{3}{2}}}=\rc{t}(1+o(1))=o(1)$.

Similarly, we have $E_3=o(1)$.

Also, we have
\begin{align*}
    E_2=\fc{2^{-\fc{3}{2}}}{\sq{\pi}}\int_x^{t-x}\lr{\fc{t}{r}}^{\rc{2}}\lr{\fc{1}{t-r}}^{\fc{3}{2}}dr=O(x^{-\rc{2}}).
\end{align*}

Hence, we have
\begin{equation}
    \lim_{t\to \iy}H(t)=0.
\end{equation}
 
 \item In the case of $\be<\be_c$:
 
By Lemma
 \ref{lem:asyofR},  we have $$R(t)\sim_{t\uparrow \iy} C_\be e^{2s_\be t},$$
where
$C_\be=\fc{\be (cm(s_\be)+1)}{2\be-cm'(s_\be)}$.
 
 By Lemma \ref{lem:asy_mgf} and  Lemma  \ref{lem:asyofR}, the first term $R^{-1}(t)\E[\si e^{2\si t}]$ in $H(t)$ converges to $0$ as $t\to \infty$.

Note that we have for $x\to \iy, t\to \iy, x/t\to 0$
\begin{align*}
    C_\be^{-1}e^{-2s_\be t}\int_0^t R(r) \E[\si e^{2\si (t-r)}] dr&=
     C_\be^{-1}e^{-2s_\be t}\Bigg(\int_0^x R(r) \E[\si e^{2\si (t-r)}] dr\\
     &+\int_x^{t-x} R(r) \E[\si e^{2\si (t-r)}] dr
     +\int_{0}^x R(t-r) \E[\si e^{2\si r}] dr\Bigg)\\
     &=:F_1+F_2+F_3
\end{align*}
 
Then we have
\begin{align*}
    F_1&=2^{-\fc{3}{2}}C_\be^{-1}\int_0^x R(r)e^{-4r}\fc{1}{(t-r)^{\fc{3}{2}}}e^{-2(s_\be-2)t}\lr{\rc{2\sq{\pi}}+o(1)}dr=o(1)
\end{align*}
where it follows from $\fc{t^{\rc{2}}}{(t-r)^{\fc{3}{2}}}=\rc{t}(1+o(1))=o(1)$ and $s_\be>2$.

Similarly, we have $F_3=o(1)$.

Also, we have
\begin{align*}
    F_2=\fc{2^{-\fc{3}{2}}}{\sq{\pi}}\int_x^{t-x}e^{-2(s_\be-2)(t-r)}\lr{\fc{1}{t-r}}^{\fc{3}{2}}dr=O(x^{-\rc{2}}).
\end{align*}

Hence, we have
\begin{equation}
    \lim_{t\to \iy}H(t)=0.
\end{equation}

\end{enumerate}

\end{proof}

\subsection{Proof of Theorem \ref{thm:ntoin_2}}

Recall that $f'$ is non-negative and Lipschitz as defined in \eqref{sde1}. Recall that $K=K(t,t)$ is defined in \eqref{eq: def_kn}. Define \begin{equation}
R_\tau^\theta(K):=e^{-\int_\tau^\theta f'(K(s))ds}\end{equation} and
\begin{equation}
    DR_\tau^\theta(K)=\frac{d}{d\tau}R_\tau^\theta(K)=f'(K(\tau,\tau))e^{-\int_\tau^\theta f'(K(s))ds}
\end{equation}

We have the following bound on $R_\tau^\theta(K)$ and $DR_\tau^\theta(K)$.

\begin{lemma}\cite[Theorem 5.3]{arous2001aging}\label{lem: bound_gba}
    Recall that we define $f', R_\tau^\theta(K),$ and $DR_\tau^\theta(K)$ as above. Then we have
    \begin{enumerate}
        \item for any $0\le \tau \le \theta \le T$ and $K\in \mathbf{C}([0,T]^2)$, we have
        \begin{equation}
            0\le R_\tau^\theta(K)\le 1, \, \mbox{ and } \int_0^t |R_\tau^\theta(K)|d\tau \le 1.
        \end{equation}
    \item for every $\theta\le T$, we have
    \begin{equation}
        \sup_{\tau \le \theta}|R_\tau^\theta(K)-R_\tau^\theta(\widetilde{K})|\le \|f'\|_{\mathcal{L}}\int_0^\theta |K(s,s)-\widetilde{K}(s,s)|ds,
    \end{equation}
where $\|f'\|_{\mathcal{L}}$ is the Lipschitz norm of $f'$.
\item for any $0\le \tau \le \theta \le T$,
\begin{align*}
        |DR_\tau^\theta(K)-DR_\tau^\theta(\widetilde{K})|\le \|f'\|_{\mathcal{L}}\Bigg\{ |K(\tau, \tau)-\widetilde{K}(\tau,\tau)|+\lr{DR_\tau^\theta(K)+DR_\tau^\theta(\widetilde{K})}\int_0^\theta |K(s,s)-\widetilde{K}(s,s)|ds\Bigg\}.
\end{align*}
    \end{enumerate}
\end{lemma}

Consider the following collections of functions with domain space $\R^2\times \mathbf{C}([0, T])$ for $T>0$ and range space one of $\mathbf{C}([0, T]^j)$ for $j=1,2,3$:
\begin{align*}
    \mathcal{G}_1:=\{g_j, j=1,2,3: g_1(Y_0, \si, B_{\cdot})(t)=\si e^{\si t}(Y_0)^2, g_2(\cdot)(t)=\si B_t^2, g_3(\cdot)(t)=\si Y_0 B_t\}.
\end{align*}
\begin{align*}
    \mathcal{G}_2:=\{g_j, j=4,5: g_4(Y_0, \si, B_{\cdot})(s, t)=\si Y_0B_s e^{\si t}, g_5(\cdot)(s,t)=\si^2 Y_0B_se^{\si t}\}.
\end{align*}
\begin{align*}
    \mathcal{G}_3:=\{g_j, j=6, 7: g_6(Y_0, \si, B_{\cdot})(u,v,t)=\si B_uB_ve^{\si t}, g_7(\cdot)(u,v,t)=\si^2 B_uB_ve^{\si t}\}.
\end{align*}

Then our collection of functions is 
\begin{equation}
    \mathcal{G}= \mathcal{G}_1\cup  \mathcal{G}_2\cup  \mathcal{G}_3.
\end{equation}

Define the empirical measure 
\begin{equation}
    \nu_T^N:=\frac{1}{N}\sum_{i=1}^N \delta_{Y_0^i, \si^i, B^i_{[0,T]}}.
\end{equation}

Define for $g\in \mathcal{G}$
\begin{equation}\label{eq:empirical_measure}
    \mathcal{C}_N:=\int g(Y_0, \si, B_{\cdot})d \nu_T^N(Y_0, \si, B_{\cdot})= \frac{1}{N}\sum_{i=1}^N g(Y_0^i, \si^i, B^i_{\cdot}),
\end{equation}
where note that for $g\in \mathcal{G}_j$, $\int g d\nu_T^N \in \mathbf{C}([0,T]^j)$ for $j=1,2,3$.

\begin{proof}[Proof of Theorem \ref{thm:ntoin_2}]

\begin{enumerate}

\item Existence and uniqueness of the limit

Apply Ito's formula from \cite[Theorem 7.6.1]{durrett2019probability}, we have
\begin{equation}
    Y_t^i=e^{\int_0^t\lr{\sigma^i-f'(K_N(r))}dr}Y_0^i+\beta^{-1/2}\int_0^te^{\int_s^t\lr{\si^{i}-f'(K_N(r))}dr}dB_s^i.
\end{equation}

Define $F_t(K,\si)=f'((K(t,t))-\si$. By integration by part, we have
\begin{equation}
    Y_t^i=e^{-\int_0^t F_r(K_N, \si^i)dr}Y_0^i+\be^{-1/2}B_t^i+\lr{-\be^{-1/2}\int_0^t B_s^i F_s(K_N,\si^i)e^{-\int_s^t F_r(K_N,\si^i)dr}ds}.
\end{equation}

Then we have
\begin{equation}
     Y_t^i=\underbrace{R_0^t(K_N)e^{\si^i t}Y_0^i}_{=:T_1^i(t)}+\underbrace{\be^{-1/2}B_t^i}_{=:T_2^i(t)}+\underbrace{\lr{-\be^{-1/2}\int_0^t B_s^i e^{\si^i(t-s)}(DR_s^t(K_N)-\si^i R_s^t(K_N))ds}}_{=:T_3^i(t)}.
\end{equation}

We denote the sum of three terms as $T_1^i(t)$, $T_2^i(t)$, and $T_3^i(t)$, respectively. In this case, 
\begin{equation}
    Y_t^i=T_1^i(t)+T_2^i(t)+T_3^i(t).
\end{equation}

Then the energy $H_N(t)$ becomes
\begin{equation}
    H_N(t)=\frac{1}{N}\sum_{i=1}^N \si^i\lr{\sum_{j=1}^3 (T_j^i(t))^2+2\sum_{j\neq k}T_j^i(t)T_k^i(t)},
\end{equation}
Plug into the expression of $T_1^i(t), T_2^i(t), T_3^i(t)$ 
\begin{align*}
    H_N(t)&=\frac{1}{N}\sum_{i=1}^N \si^i(R_0^t(K_N))^2e^{2\si^i t}(Y_0^i)^2+\frac{1}{N}\sum_{i=1}^N \si^i\be^{-1}(B_t^i)^2\\
    &+\frac{\beta^{-1}}{N}\sum_{i=1}^N \si^i\int_0^t\int_0^t B_{u}^i B_{v}^ie^{\si^i(2t-u-v)}\lr{DR_{u}^t(K_N)-\si^iR_{v}^t(K_N)}\lr{DR_{v}^t(K_N)-\si^iR_{v}^t(K_N)}dudv\\
    &+\frac{2}{N}\sum_{i=1}^N\si^i\Bigg\{\be^{-1/2}Y_0^iB_t^iR_0^t(K_N)-\beta^{-1/2}\int_0^t Y_0^iB_s^i e^{\si^i(2t-s)}R_0^t(K_N)\lr{DR_s^t(K_N)-\si^i R_s^t(K_N)}ds\\
    &-\beta^{-1}\int_0^t B_t^iB_s^ie^{\si^i(t-s)}\lr{DR_s^t(K_N)-\si^i R_s^t(K_N)}ds\Bigg\}
\end{align*}

Hence, the above equation of $H_N(t)$ specifies the function  $$\Phi: \mathbf{C}([0,T]^2)\times \mathbf{C}([0,T])^{|\mathcal{G}_1|}\times \mathbf{C}([0,T]^2)^{|\mathcal{G}_2|}\times \mathbf{C}([0,T]^3)^{|\mathcal{G}_3|}\to \mathbf{C}([0,T])$$ such that
    $$H_N=\Phi(K_N, \mathcal{C}_N).$$

By Lemma \ref{lem: bound_gba}, for any $\mathcal{C}_N$ defined as in \eqref{eq:empirical_measure} and $K_N, \widetilde{K}_N\in \mathbf{C}([0,T]^2)$, there exists a constant $C_1>0$ so that
\begin{equation}\label{eq:upperbound_K}
    \sup_{0\le t\le T}|\Phi(K_N, \mathcal{C}_N)-\Phi(\widetilde{K}_N, \mathcal{C}_N)|\le C_1\int_0^t |K_N(s,s)-\widetilde{K}_N(s,s)|ds
\end{equation}

Similarly, we apply Lemma \ref{lem: bound_gba}, then for any $\mathcal{C}_N, \widetilde{\mathcal{C}}_N$ defined as in \eqref{eq:empirical_measure} and $ \widetilde{K}_N\in \mathbf{C}([0,T]^2)$, there exists a constant $C_2>0$ so that
\begin{equation}\label{eq:upperbound_C}
    \sup_{0\le t\le T}|\Phi(\widetilde{K}_N, \mathcal{C}_N)-\Phi(\widetilde{K}_N, \widetilde{\mathcal{C}}_N)|\le C_2\|\mathcal{C}_N-\widetilde{\mathcal{C}}_N\|_\iy
\end{equation}

By Theorem \ref{thm:ntoin}, we know that $K_N$ converges to the deterministic limit $K$ almost surely and $K$ is unique. By \cite[Lemma 3.7]{liang2022high}, each element in $\mathcal{C}_N$ converges to deterministic limits $\mathcal{C}$ almost surely under the i.i.d. initial conditions.

Combining \eqref{eq:upperbound_K} and \eqref{eq:upperbound_C}, then we have
\begin{align*}
    \|\Phi(K_N, \mathcal{C}_N)-\Phi(K, \mathcal{C})\|_\iy &= \|\Phi(K_N, \mathcal{C}_N)-\Phi(K_N, \mathcal{C})+\Phi(K_N, \mathcal{C})-\Phi(K, \mathcal{C})\|_\iy\\
    &\le \|\Phi(K_N, \mathcal{C}_N)-\Phi(K_N, \mathcal{C})\|_\iy +\|\Phi(K_N, \mathcal{C})-\Phi(K, \mathcal{C})\|_\iy\\
    &\le C_1\|K_N-K\|_\iy+C_2\|\mathcal{C}_N-\mathcal{C}\|_\iy\to 0, \mbox{ as } N\to \iy.
\end{align*}

Hence, $H_N$ converges to deterministic function $H$. 

Also, we have $H=\Phi(K, \mathcal{C})$. Indeed, 
\begin{align*}
    |H-\Phi(K, \mathcal{C})|\le |H-H_N|+|H_N-\Phi(K, \mathcal{C})|\to 0, \mbox{ as } N\to \iy.
\end{align*}

\item Equations for the limit points

Next, we characterize the limit $H$ as follows. Recall that $R_\tau^\theta(K)=e^{-\int_\tau^\theta f'(K(s))ds}$. Then $Y_t$ can be expressed as
\[Y_t^i=R_0^t(K_N)e^{\sigma^i t}Y_0^i+\be^{-1/2}\int_0^t R_s^t(K_N)e^{\si^i(t-s)}dB_s^i.
\]

Substitute the above expression of $Y_t^i$ to $H_N(t)$ defined as in \eqref{eq:energy}, then we have
\begin{align*}
    H_N(t)&=\rc{N}\sumiN \si^i e^{2\si^i t} (Y_0^i)^2 (R_0^t(K_N))^2+\fc{\be^{-1}}{N}\sumiN \si^i\lr{\int_0^t R_s^t(K_N)e^{\si^i(t-s)}dB_s^i}^2\\
    &+\fc{2\be^{-1/2}}{N}\sumiN \si^i \lr{R_0^t(K_N)e^{\si^i t}Y_0^i}\int_0^t R_s^t(K_N)e^{\si^i(t-s)}dB_s^i.
\end{align*}
As $N\to \infty$, note that the first term is convergence to its expectation as in \eqref{ex1} by strong law of large number (SLLN) \cite[Theorem 2.4.1]{durrett2019probability},  the limit  of the second term given in \eqref{ex2} is obtained by SLLN and It\^o isometry \cite[Lemma 3.1.5]{oksendal2003stochastic}, and the last term is convergence to zero by SLLN. Note that we take the limit that requires the empirical measure  $\nu_T^N:=\frac{1}{N}\sum_{i=1}^N \delta_{Y_0^i, \si^i, B^i_{[0,T]}}$ converges to the desired limits under the i.i.d. initial conditions. Hence, we obtain the desired integro-differential equation as in Theorem \ref{thm:ntoin_2}.

\end{enumerate}

\end{proof}

\subsection{Proof of \cref{coro_hk}}

\begin{proof}
    Consider first the regime $\be>\be_c$. Recall that $R(t)$ is defined as in \eqref{eq: eq_r_t} and $K(t)$ is given in Theorem \ref{thm:ntoin}. We rewrite $K(t)$ as follows:
    \begin{equation}
        K(t)=R^{-1}(t)\lr{\E[e^{2t\si}]+\be^{-1}\int_0^t R(r)\E[e^{2(t-r)\si}]dr}.
    \end{equation}
    We apply Lemma \ref{lem:asyofR} and Lemma \ref{asy_mgf_1}, and plug in the asymptotic limit of $R(t)$ and $E[e^{2t\si}]$, we get 
    \begin{equation}
        \lim_{t\to\iy}K(t)=\frac{2^{-\frac{7}{2}}\pi^{-\frac{1}{2}}(4\be+1)}{C_\be(4\be-c)}+\frac{1}{2}\be^{-1}.
    \end{equation}
    Combining this result and Theorem \ref{thm: limit_energy}, we obtained the desired result.

    For $\be<\be_c$, we will show that $\lim_{t\to\iy}K(t)=C$ for some non-zero constants $C$. Take $h(t)=R(t)K(t)$. Note that
    \begin{equation}
        R'(t)=2cK(t)R(t)=2ch(t).
    \end{equation}
    Thus, we have $2c\Ls_h(z)=z\Ls_R(z)-1$.
    
By Lemma \ref{lem:lap}, we have
\begin{equation}
    \Ls_R(z)=\frac{1+cm(z)}{2z-c\be^{-1}m(z)}
\end{equation}
Hence, we get
\begin{equation}
    \Ls_g(z)=\frac{1}{2c}\lr{\frac{czm(z)(1+\be^{-1})-z}{2z-c\be^{-1}m(z)}}.
\end{equation}
Note that  $\Ls_g(z)$ has a simple pole at $s_\be$, which is a solution to $2z=c\be^{-1}m(z)$. Thus there exists a constant $C>0$ such that
\begin{equation}
    \lim_{z\to 0}z\Ls_g(z+s_\be)=C
\end{equation}

By \cite[Lemma 7.2]{ben2003aging}, we have
\begin{equation}
    \lim_{t\to\iy}e^{-2s_\be t}h(t)=C
\end{equation}
Hence, $\lim_{t\to\iy} K(t)=C$ for some non-zero constants $C$. 
Hence, the desired follows from Theorem \ref{thm: limit_energy}.

For $\be=\be_c$, we apply the same proof as $\be<\be_c$. Note that $\lim_{t\to \iy}K(t)=C_1$ for some non-zero constants $C_1$. Hence, we still obtain the desired result by  Theorem \ref{thm: limit_energy}.
\end{proof}

%%%%%%%%%%%%%%%%%%%%%%%%%%%%%%%%%%%%%%%%%%%%%%%%
\section{Gradient descent in spherical SK model: hitting time analysis
}\label{sec:hit_time_gd}

In this section, we mainly study the algorithm complexity of using the gradient descent algorithm to find the extreme eigenvalues of the Wigner matrix. This is related to our previous main results about the asymptotic limit of energy in  \cref{sec:long_time} for the description of the time to the equilibrium state of the SSK model.

Recall that we previously defined the Spherical Sherrington-Kirkpatrick (SSK) model on a sphere of radius $N$. For simplicity, we consider the SSK model on the unit sphere in this section characterized by the Hamiltonian
\begin{equation*}
    H_{\bJ}(X)=X^T\bJ X, 
\end{equation*}
where $X=(X_1,\dots, X_N)$ with $\|X\|_2=1$  and $\bJ$ is the normalized symmetric Wigner matrix.

We define $N$ eigenvalues of $\bJ$ in increasing order as $$\lam_1\le \lam_2\le \dots \le \lam_N,$$ and $v_1, v_2, \dots, v_N$ be an orthonormal basis of eigenvectors of $\bJ$ so that $\bJ v_i=\lam_iv_i$ for $i=1,\dots, N$.

The gradient descent algorithm (a.k.a., zero-temperature dynamics) of the SK model on the unit sphere is defined as follows:
\begin{equation}\label{eq:sde_sphereical_simple}
    dX_t=-\nabla_{\bS^{N-1}} H_{\bJ}(X_t)dt,
\end{equation}
where the initial data $X_0=\{X_0^i\}_{1\le i\le N}$ is uniformly distributed on the unit sphere $\bS^{N-1}$, and
$\nabla_{\bS^{N-1}}$ is the gradient on the unit sphere $\bS^{N-1}$ and $$\nabla_{\bS^{N-1}} f(x):= \na f(x)-(\na f(x)\ct x)x, \, x\in \R^N$$ for smooth functions $f$.

Before presenting the main result of this section, we need to add one assumption on the Wigner matrix $\bJ$ to ensure that the Tracy-Widom law holds, which provides information about the rescaling of the largest eigenvalue around the edge in \cite{tracy1994level,tracy1996orthogonal}. Moreover, Lee and Yin proved that this is not only true for Gaussian ensembles but also holds for more general situations in \cite{lee2014necessary}. A simple sufficient criterion for Tracy-Widom law has been proved as follows.

\begin{thm}\cite[Theorem 1.2]{lee2014necessary}\label{thm_yinjun}
Let $\bJ$ be the normalized Wigner matrix as in \cref{def_wigner} and $\lam_1\le \lam_2\le \dots \lam_N$ the eigenvalues of $\bJ$ as before.  If the off-diagonal entry of the Wigner matrix satisfies
\begin{equation}\label{eq:criterion}
    \lim_{s\to \iy} s^4\P(|Z_{12}|\ge s)=0,
\end{equation}
then the joint distribution function of $k$ rescaled the largest eigenvalues
\begin{equation}
    \P(N^{2/3}(\lam_N-2)\le s_1, N^{2/3}(\lam_{N-1}-2)\le s_2, \dots, N^{2/3}(\lam_{N-k+1}-2)\le s_k)
\end{equation}
has a limit as $N\to \iy$, which coincides with that in the GUE and GOE cases, i.e., it weakly converges to the Tracy-Widom distribution. This result also holds for the smallest eigenvalues $\lam_1, \dots, \lam_k$.
\end{thm}

\begin{remark}
Note that any distribution with a finite fourth moment satisfies the criterion \eqref{eq:criterion}, however, the converse statement does not hold. See \cite[Remark 1.3]{lee2014necessary} for a counterexample. 
\end{remark}

Fix $\ep\in (0,1)$. Denote by $T_\ep$ the hitting time of the overlap between the output $X_t$ of the gradient descent and eigenvector $v_1$ corresponding to the smallest eigenvalue of $\bJ$ is greater than $\ep$, that is $$T_\ep:=\inf_{t>0}\{|v_1\cdot X_t| \ge \ep\}.$$ 

Our main result is the lower bound and upper bound of the hitting time $T_\ep$ as follows.

\begin{thm}\label{thm:lower_bound_gd}
Assume that the normalized $N\times N$ Wigner matrix $\bJ=\{J_{ij}\}_{1\le i,j\le N}=\{\frac{Z_{ij}}{\sqrt{N}}\}_{1\le i,j\le N}$ obeys the following condition: for $1\le i,j \le N$
    \begin{equation}\label{eq:c1c}
        \E|Z_{ij}|^{4}\le C
    \end{equation}
    for some constants $C>0$. Consider the gradient descent described in \eqref{eq:sde_sphereical_simple} with the same setting as before. Fix $\ep\in (0,1)$. Let the hitting time $T_\ep$ be defined as before.
For every $\delta>0$, there exist constants $C_1=C_1(\delta)>0$, $C_2=C_2(\delta)>0$, and $C_3=C_3(\ep, \delta)>0$ such that
\begin{equation*}
    \lim_{N\to \iy}\P\lr{C_1 N^{2/3}<T_\ep<C_2 N^{2/3}\log(C_3 N)}> 1-\delta.
\end{equation*}

\end{thm}

Combining Theorem \ref{thm_yinjun} and continuous mapping theorem in \cite[Theorem 3.2.10]{durrett2019probability} we obtain the following Lemma. Moreover, we have $\lam_2-\lam_1=O_p(N^{-2/3})$. 

\begin{lemma}\cite[Lemma 2.3]{landon2022free}\label{lem:asymp_small_two}
    Let $\bJ$ be the normalized Wigner matrix. Assume that $\bJ$ obeys the same condition as in \eqref{eq:c1c}. Let $\lam_1<\lam_2$ be the two smallest eigenvalues of $\bJ$. Then for every $\ep>0$, there exists $\delta>0$ so that 
    \begin{equation}
        \P\lr{N^{2/3}(\lam_2-\lam_1)\ge \delta}\ge 1-\ep
    \end{equation}
   for all $N$ large enough. This result also holds for the two largest eigenvalue $\lam_{N-1}<\lam_N$.
\end{lemma}

Define the overlap of the output $X_t$ and eigenvectors $v_i$ of $\bJ$ by $h_i(t):=v_i\cdot X_t$ for $i=1,\dots, N$. Note that we can solve $h_i(t)$ for $i=1,\dots, N$ as follows.

\begin{lemma}\label{lem:ode}
    Assume that the same setting holds as in Theorem \ref{thm:lower_bound_gd}. Then we have
        \begin{equation}
    |h_j(t)|=\fc{|h_j(0)|e^{-2\lam_j t}}{\sq{\sum_{i=1}^N h^2_i(0) e^{-4\lam_i t}}}.
\end{equation}
\end{lemma}

\begin{proof}
    
Note that the gradient descent  \eqref{eq:sde_sphereical_simple} can be simplified as follows 
\begin{equation}\label{eq:sde_gd_ecu}
    dX_t=-\lr{\na H_{\bJ}(X_t)-(\na H_{\bJ}(X_t)\cdot X_t)X_t }dt=-2\bJ X_tdt+2H_{\bJ}(X_t) X_tdt.
\end{equation}

By spectral decomposition we have $$\bJ=\sum_{i=1}^N v_iv_i^T\lam_i,$$ where $v_1, \dots, v_N$ are eigenvectors corresponding to eigenvalues $\lam_1\le \lam_2\le \dots \le \lam_N$ of $\bJ$.

Note that 
\[H_{\bJ}(X_t)=X_t\cdot \bJ X_t=\sum_i \lambda_i (h_i(t))^2.
\]

Consider the dot product $v_1$ on the both sides of  \eqref{eq:sde_gd_ecu} and substitute the above equation of $H_{\bJ}(X_t)$ to get 
\begin{align*}
    \rc{2}h_1'(t)&=h_1(t)H_{\bJ}(X_t)-v_1\cdot (\sum_i v_iv_i^T\lam_i)X_t\\
    &=h(t)H_{\bJ}(X_t)-\lam_1 v_1\cdot X_t\\
    &=h_1(t)H_{\bJ}(X_t)-\lam_1 h_1(t).
\end{align*}

By the fact that $\sum_i h_i(t)^2=1$ we have
\begin{equation}\label{eq:eq_h_t}
     \rc{2}h_1'(t)=(H_{\bJ}(X_t)-\lam_1) h(t)=\sum_{i=1}^N[(\lam_i-\lam_1)h_i^2]h_1(t).
\end{equation}
where we write $h_i(t)=h_i$ for convenience, $i=1,\dots, N$. 

Similarly, we have
\begin{equation}\label{eq:eq_t_hj}
    \rc{2} h_j'(t)=\sum_{i=1}^N[(\lam_i-\lam_j)h_i^2]h_j(t).
\end{equation}

Multiply $h_j(t)$ on the both side of \eqref{eq:eq_t_hj} yields
\begin{align*}
    (h^2_j(t))'=4\sum_{i=1}^N \lr{(\lam_i-\lam_j)h_i^2}h_j^2
\end{align*}

Denote $f(t)=\sum_i\lam_i h_i^2(t)$ and $F(t)=\int_0^t f(s)ds$. 

Both sides of the equation are divided by $h_j^2$ yields for$j=1,\dots, N$
\begin{align*}
    (\log h_j^2(t))'=4f(t)-4\lam_j
\end{align*}

Integrating the two sides with respect to time $t$ yields
\begin{equation}\label{eq:repren_hi}
    h^2_j(t)=h_j(0)^2\exp\lr{4F(t)-4\lam_j t}
\end{equation}

Taking the derivative with respect to both sides of $F(t)$ and substitute \eqref{eq:repren_hi} yields
\begin{align*}
    F'(t)=\sum_i \lam_i h_i^2(t)=e^{4F(t)}\lr{\sum_i \lam_i h^2_i(0) e^{-4\lam_t t}}.
\end{align*}
Both sides are divided by the integration factor $e^{4F(t)}$ and integrating with respect to $t$ to obtain
\begin{equation}
    e^{-4F(t)}-1=\sum h_i^2(0)(e^{-4\lam_i t}-1).
\end{equation}
So we get 
\begin{equation}\label{eq:represen_Ft}
    F(t)=-\rc{4}\log\lr{\sum_i h_i^2(0) e^{-4\lam_i t}}.
\end{equation}

Substituting  \eqref{eq:represen_Ft} into equation \eqref{eq:repren_hi} yields for $j=1,2,\dots, N$
\begin{equation}\label{eq:final_rep_hit}
    |h_j(t)|=\fc{|h_j(0)|e^{-2\lam_j t}}{\sq{\sum_{i=1}^N h^2_i(0) e^{-4\lam_i t}}}.
\end{equation}

\end{proof}

To prove Theorem \ref{thm:lower_bound_gd}, we need the following Lemma.
\begin{lemma}\label{lem:ine_sym}
    Consider a sequence of i.i.d. positive random variables $X_1,\dots, X_k$. For every constant $C>0$,  we have
    $$ 
\P\left(\frac{X_1+X_2+\dots+X_k}{X_j}>C\right)> 1-\frac{C}{k},  \mbox{ for }\, j=1,\dots, k.
$$
\end{lemma}

\begin{proof}
    Note that
    \begin{equation*}
        0<\E\left[\frac{X_1}{\sum_{i=1}^k X_i}\right]=\E\left[\frac{X_2}{\sum_{i=1}^k X_i}\right]=\dots=\E\left[\frac{X_k}{\sum_{i=1}^k X_i}\right]<1.
    \end{equation*}
    Then we have $$\E\left[\frac{X_1}{\sum_{i=1}^k X_i}\right]=\frac{1}{k}.$$

    By Markov's inequality, we get for every constant $C>0$
    \begin{equation}
        \P\left(\frac{X_1+X_2+\dots+X_k}{X_j}\le C\right)= \P\left(\frac{X_j}{X_1+X_2+\dots+X_k}\ge 1/C\right)\le C\cdot \EX{\frac{X_j}{\sum_i X_i}}= \frac{C}{k}.
    \end{equation}
\end{proof}

We require the following Lemma.

\begin{lemma}\label{lem:tail_gaussian}
    Let $X$ be a standard normal random variable. Then for every $\epsilon>0$, there exists a constant $\delta>0$ so that
    \begin{equation}
        \P(|X|>\delta)\ge 1-\epsilon.
    \end{equation}
\end{lemma}

\begin{proof}
    For every $\epsilon>0$, there exists a sufficiently small constant $\delta \in\lr {0, \sq{\frac{\pi}{2}}\ep}$ so that 
    \begin{equation}
        \P(|X|\le \delta)=\frac{1}{\sq{2\pi}}\int_{-\delta}^\delta e^{-x^2/2}dx\le \frac{2}{\sq{2\pi}}\delta<\epsilon,
    \end{equation}
    where the above inequality follows from the fact that $e^{-x^2/2}\le 1$ for $x\in \R$.
\end{proof}

Armed with the previous results, we then can prove our main result.

\begin{proof}[The proof of Theorem \ref{thm:lower_bound_gd}]

By Lemma \ref{lem:ode}, we have
\begin{equation}\label{eq:rep_h_10}
     |h_1(t)|=\fc{|h_1(0)|e^{-2\lam_1 t}}{\sq{\sum_i h^2_i(0) e^{-4\lam_i t}}}=\frac{|h_1(0)|}{\sq{h_1^2(0)+\sum_{i=2}^N h^2_i(0) e^{-4(\lam_i-\lam_1) t}}}
\end{equation}

Next, we consider the upper and lower bound of the hitting time $T_\ep$, respectively.

\begin{enumerate}
    \item Lower bound of $T_\ep$.

For any $\delta>0$, we fix the first $k$ terms of the denominator of \eqref{eq:rep_h_10} and then we will find desired $k$ below depending on $\epsilon$ and $\delta$ (independent of $N$):
\begin{equation}
    |h_1(t)|\le|h_1(0)|\lr{h_1^2(0)+\sum_{i=2}^k h_i^2(0) e^{-4t(\lambda_i-\lambda_1)}}^{-1/2}.
\end{equation}

By the fact that $(\lambda_i-\lambda_1)\le (\lambda_k-\lambda_1)$ for $i=1,\dots, k-1$, then we upper bound the above inequality:
\begin{equation}
    |h_1(t)|\le |h_1(0)|\lr{h_1^2(0)+e^{-4t(\lambda_k-\lambda_1)}\sum_{i=2}^k h_i^2(0)}^{-1/2}
\end{equation}

As $t\ge T_\epsilon$, we have 
\begin{equation}
    \epsilon\le |h_1(t)|\le |h_1(0)|\lr{h_1^2(0)+e^{-4t(\lambda_k-\lambda_1)}\sum_{i=2}^k h_i^2(0)}^{-1/2}.
\end{equation}
Then we get
\begin{equation}
    T_\epsilon \ge \frac{1}{4(\lambda_k-\lambda_1)}\log\left(\frac{h_2^2(0)+\dots + h_k^2(0)}{h_1^2(0)(\epsilon^{-2}-1)}\right)
\end{equation}

For any $\delta>0$, we apply Lemma \ref{lem:ine_sym} and choose $C=2\ep^{-2}-1$:
\begin{align*}
     \lim_{N\to \infty}\mathbb{P}\left(\frac{h_2^2(0)+\dots + h_k^2(0)}{h_1^2(0)(\epsilon^{-2}-1)}>2\right)&=\lim_{N\to \infty}\mathbb{P}\left(\frac{h_1^2(0)+h_2^2(0)+\dots + h_k^2(0)}{h_1^2(0)}>\frac{2}{\epsilon^2}-1\right)\\
     &\ge 1-\delta/2,
\end{align*}
where we take $k=[2(2\epsilon^{-2}-1)/\delta]+1$.

%By Theorem \ref{thm_yinjun} and continuous mapping theorem, for any $\delta>0$ there exists $c_1>0$ so that 
%\begin{equation}
%    \P(\lam_k-\lam_1>c_1N^{-2/3})\ge 1-\delta/2
%\end{equation}

By the similar argument of Lemma \ref{lem:asymp_small_two}, for any $\delta>0$ there exists a constant $c_1=c_1(\delta)>0$ so that
\begin{equation}
     \lim_{N\to \infty}\P(N^{2/3}(\lam_k-\lam_1)<c_1)\ge 1-\delta/2
\end{equation}
%$\lam_k-\lam_1=O_p(N^{-2/3})$.

Hence, for any $\delta>0$ there exists $c_2=\frac{\log 2}{4c_1}$ so that
\begin{equation}
    \lim_{N\to \iy}\P(T_\ep>c_2 N^{2/3})\ge 1-\delta.
\end{equation}

% By Theorem \ref{thm_yinjun} and continuous mapping theorem, we have $\lam_2-\lam_1=O_p(N^{-2/3})$. Since $X_0$ is uniformly distributed on the sphere, then we have $h_2(0)/h_1(0)$ converges to the standard Gaussin distribution as $N\to \iy$. So we get $h_2(0)/h_1(0)=O_p(1)$.

\item Upper bound of $T_\ep$.

Note that $(\lambda_2-\lam_1)\le (\lam_i-\lam_1)$ for $i=3,\dots, N$. We upper bound each term of the denominator in \eqref{eq:rep_h_10} by $e^{-4t(\lam_i-\lam_1)}\le e^{-4t(\lam_2-\lam_1)}$ for $i=3,\dots, N$. Then we have
\begin{equation}
    |h_1(t)|\ge \frac{|h_1(0)|}{\sq{h_1^2(0)+e^{-4(\lam_2-\lam_1) t}\sum_{i=2}^N h^2_i(0)}}=\frac{|h_1(0)|}{\sq{h_1^2(0)+e^{-4(\lam_2-\lam_1) t}(1-h_1^2(0))}}
\end{equation}

For $t\le T_\ep$, we get
\begin{equation}
    \ep\ge |h_1(t)|\ge \frac{|h_1(0)|}{\sq{h_1^2(0)+e^{-4(\lam_2-\lam_1) t}(1-h_1^2(0))}}.
\end{equation}
and this yields 
\begin{equation}
    T_\ep\le \frac{1}{4(\lam_2-\lam_1)}\log\lr{\frac{h_1^{-2}(0)-1}{\ep^{-2}-1}}.
\end{equation}

By Lemma \ref{lem:asymp_small_two}, for any $\delta>0$ there exists a constant $c_3=c_3(\delta)>0$ so that 
\begin{equation}\label{ineq:gap_upp}
    \lim_{N\to \iy}\P\lr{N^{2/3}(\lam_2-\lam_1)\ge c_3}\ge 1-\delta/2
\end{equation}

%\textcolor{red}{[note: need to upper bound the term inside log]}

Note that $\sq{N}h_1(0)$ is asymptotic Gaussian by \cite[Theorem 13]{tao2012random}. By Lemma \ref{lem:tail_gaussian}, for every $\delta>0$, there exists a constant $c_4>0$ so that 
\begin{equation}\label{eq:upper_gau_tail}
    \lim_{N\to \iy}\P\lr{\sq{N}|h_1(0)|>c_4}\ge 1-\delta/2.
\end{equation}

For any $\delta>0$, we take $c_5=\frac{1}{c_4^2(\ep^{-2}-1)}>0$ and then get
\begin{align*}
    \P\lr{\frac{h_1^{-2}(0)-1}{\ep^{-2}-1}<c_5 N}=\P\lr{|h_1(0)|>\sq{\frac{c_4}{c_4^2+N}}} =\P\lr{\sq{1+\frac{c_4^2}{N}}\sq{N}|h_1(0)|>c_4}
\end{align*}

By inequality \eqref{eq:upper_gau_tail}, for every $\delta>0$ we have
\begin{equation}\label{eq: upper_b_h1}
    \lim_{N\to \iy}\P\lr{\frac{h_1^{-2}(0)-1}{\ep^{-2}-1}<c_5 N}\ge 1-\delta/2.
\end{equation}

Combining the two upper bounds \eqref{ineq:gap_upp} and \eqref{eq: upper_b_h1}, for every $\delta>0$ there exist constants $c_5$ defined as above and $c_6=\frac{1}{4c_3}$ so that
\begin{equation}
    \lim_{N\to \iy}\P\lr{T_\ep<c_6N^{2/3}\log(c_5 N)}\ge 1-\delta.
\end{equation}
\end{enumerate}
\end{proof}

\section{Power iteration method: a hitting time perspective
}\label{sec:ht_pi}

In this section, we conduct an average case analysis of the hitting time for the top eigenvectors using the power iteration method \cite[Section 9.3]{burden1997numerical}, paralleling the approach in \cref{sec:hit_time_gd}. Our focus will be on a normalized Wigner matrix, for which we will arrange the eigenvalues by their absolute values. 

Let $\bJ$ be a normalized Wigner matrix as before. We arrange the eigenvalues of the matrix $\bJ$ in order of their absolute values
$$
|\sigma_N|\le |\si_{N-1}|\le \dots \le |\si_2|\le |\si_1|
$$ with corresponding eigenvectors $u_N, u_{N-1},\dots, u_2, u_1$. Power iteration is then employed to estimate the matrix's dominant eigenvalue $\sigma_1$ and the corresponding eigenvector.

Given an arbitrary initial vector $q_0\in \R^N$ with $\|q_0\|_2=1$, we take for $k=1,2,3,\dots$
\begin{equation}\label{eq:def_qk}
    q_k:=\frac{\bJ q_{k-1}}{\|\bJ q_{k-1}\|_2}
\end{equation}

Note that there are some constants $\al_i$ for $i=1,2,\dots, N$ so that 
\begin{equation}
    q_0=\sum_{i=1}^N \al_i u_i.
\end{equation}

%By \cite[Section 9.3]{burden1997numerical}, we have
%\begin{equation}
%    q_k=\frac{\bJ^k q_{0}}{\|\bJ^k q_{0}\|_2}=\text{sign}(\si_1^k)\frac{\al_1u_1+w_k}{\|\al_1u_1+w_k\|_2},
%\end{equation}
%where we let $w_k:=\sum_{i=2}^N \al_i \lr{\frac{|\si_i|}{|\si_1|}}^ku_i$

Then we have for for $k=1,2,\dots N$,
\begin{equation}\label{eq:jkq_eq}
     \bJ^kq_0=\lr{\sum_{i=1}^N u_iu_i^T \si_i^k}\lr{\sum_{i=1}^N \al_i u_i}=\si_1^k\lr{\al_1u_1+\sum_{i=2}^N \al_i\lr{\fc{\si_i}{\si_1}}^ku_i}.
\end{equation}

Our main theorem in this section is as follows.

\begin{thm}\label{thm:hit_power}
   Let $\bJ$ be the $N\times N$ normalized Wigner matrix. Assume that $\bJ$ obeys the same condition as in \eqref{eq:c1c}. Fix $0<\ep<1$ and intial value $q_0\in \R^N$.  Let $q_k$ be output of power iteration for step $k\ge 1$ and $T_\ep=\inf_{k\ge 1}\{\la{q_k\cdot u_1}\ge \ep\}$. For every $\delta>0$, there exists constants $C_1, C_2, C_3>0$ so that
   \begin{equation}
       \lim_{N\to \infty}\P\lr{C_1N^{2/3}<T_\ep<C_2 N^{2/3}\log (C_3N)}>1-\delta.
   \end{equation}
\end{thm}

To prove our main theorem, we require the following lemma results.

\begin{lemma}\label{lem:ratio_eigen}
     Let $\bJ$ be the $N\times N$ normalized Wigner matrix. Assume that $\bJ$ obeys the same condition as in \eqref{eq:c1c}. Ordering its eigenvalues $|\si_N|\le |\si_{N-1}|\le \dots \le |\si_1|$. Then we have $|\si_1|/|\si_2|=1+\mathcal{O}(N^{-2/3})$
\end{lemma}

\begin{proof}
    By Theorem \ref{thm_yinjun}, we assume that $|\si_1|=2+\gamma_1$ and $|\si_2|=2+\gamma_2$ with $\gamma_1, \gamma_2=\mathcal{O}(N^{-2/3})$. Then by Taylor expansion, we get
    \begin{align*}
         \frac{|\si_1|}{|\si_2|}&=\frac{2+\ga_1}{2}(1+\frac{\ga_2}{2})^{-1}=\frac{2+\ga_1}{2}\lr{1-\frac{\ga_2}{2}+\mathcal{O}(\ga_2^2)}\\
         %&=(1+\frac{\ga_1}{2})\lr{1-\frac{\ga_2}{2}+\mathcal{O}(\ga_2^2)}\\
         %&=1+\frac{1}{2}(\ga_1-\ga_2)+\mathcal{O}(\ga_2^2)+\mathcal{O}(\ga_1\ga_2)\\
         &=1+\frac{1}{2}(\ga_1-\ga_2)+\mathcal{O}(N^{-4/3})\\
         &=1+\mathcal{O}(N^{-2/3}).
    \end{align*}
\end{proof}

\begin{proof}[The proof of \cref{thm:hit_power}]
    We first show that the lower bound of $T_\ep$. Fix $\delta>0$. We will find appropriate $1\le l<N$ depending on $\ep$ and $\delta$ (independent of $N$) to lower bound the denominator of $q_k\cdot u_1$ by the first $l$ terms as follows.
    
By \eqref{eq:def_qk} and \eqref{eq:jkq_eq}, we have
\begin{equation}
     \la{q_k\cdot u_1}= \fc{\la{\si_1^k\al_1}}{\sq{\sum_{i=1}^N \al_i^2\si_i^{2k}}}\le  \fc{\la{\si_1^k\al_1}}{\sq{\sum_{i=1}^l \al_i^2\si_i^{2k}}}=\lr{1+\sum_{i=2}^l \frac{\al_i^2}{\al_1^2}\frac{\si_i^{2k}}{\si_1^{2k}}}^{-1/2}
\end{equation}

By the fact that $|\si_i|\ge |\si_l|$ for $i=1,\dots, l$, we obtain
\begin{equation}
    \la{q_k\cdot u_1}\le \lr{1+\frac{\si_l^{2k}}{\si_1^{2k}}\sum_{i=2}^l \frac{\al_i^2}{\al_1^2}}^{-1/2}
\end{equation}

For $k\ge T_\ep$, we have $\la{q_k\cdot u_1}\ge \ep$. Thus, we get
\begin{equation}
     T_\ep\ge \frac{1}{2}\log^{-1}\lr{\fc{|\si_1|}{|\si_l|}}\log\lr{\frac{\al_2^2+\dots+\al_l^2}{\al_1^2(\ep^{-2}-1)}}.
\end{equation}

Then we apply \cref{lem:ine_sym} and choose $C=2\ep^{-2}-1$:
\begin{align*}
     \lim_{N\to \infty}\mathbb{P}\left(\frac{\al_2^2+\dots + \al_l^2}{\al_1^2(\epsilon^{-2}-1)}>2\right)&=\lim_{N\to \infty}\mathbb{P}\left(\frac{\al_1^2+\al_2^2+\dots + \al_l^2}{\al_1^2}>\frac{2}{\epsilon^2}-1\right)\\
     &\ge 1-\delta/2,
\end{align*}
where we take $l=[2(2\epsilon^{-2}-1)/\delta]+1$.

By \cref{lem:ratio_eigen} and \cref{thm_yinjun}, for any $\delta>0$ there exists a constant $c_1=c_1(\delta)>0$ so that
\begin{equation}
      \lim_{N\to \iy}\P\lr{N^{2/3}\lr{\frac{|\si_1|}{|\si_l|}-1}\le c_1}\ge 1-\delta/2.
\end{equation}

By the fact that $\log(1+x)<x$ for $x>0$, then for any $\delta>0$ there exists $c_2=c_1^{-1}>0$ so that
\begin{equation}
    \lim_{N\to \iy}\P(T_\ep\ge c_2 N^{2/3})\ge 1-\delta.
\end{equation}

Next, we prove the upper bound of $T_\ep$. 

By inequality that $|\si_i|\le |\si_2|$ for $i=2,3,\dots, N$, we have
\begin{equation}
    \la{q_k\cdot u_1}=\lr{1+\sum_{i=2}^N \frac{\al_i^2}{\al_1^2}\frac{\si_i^{2k}}{\si_1^{2k}}}^{-1/2}\ge  \lr{1+ \frac{\si_2^{2k}}{\si_1^{2k}}\sum_{i=2}^N \frac{\al_i^2}{\al_1^2}}^{-1/2}.
\end{equation}

Since we have $\|q_0\|_2=1$, then 
\begin{equation}
     \la{q_k\cdot u_1}\ge \lr{1+ \frac{\si_2^{2k}}{\si_1^{2k}} \frac{1-\al_1^2}{\al_1^2}}^{-1/2}.
\end{equation}

By Bernoulli's inequality  $(1+x)^\al\ge 1+\al x$ for $\al\le 0$ and $x> -1$, we get 
\begin{align*}
     \la{q_k\cdot u_1}\ge 1-\rc{2}\lr{\fc{\si_2}{\si_1}}^{2k}\lr{\fc{1-\al_1^2}{\al_1^2}}
\end{align*}

Note that for $k\le T_\ep$, we have $\la{q_k\cdot u_1} \le \ep$.  Then we get
\begin{equation}
   T_\ep\le \fc{2}{\log\la{\fc{\si_1}{\si_2}}}\log\lr{\fc{\al_1^{-2}-1}{2(1-\ep)}}.
\end{equation}

By \cref{lem:ratio_eigen}, for any $\delta>0$ there exists a constant $c_3=c_3(\delta)>0$ so that
\begin{equation}\label{eq:ineq_ratio_e}
      \lim_{N\to \iy}\P\lr{N^{2/3}\lr{\la{\frac{\si_1}{\si_2}}-1}\ge c_3}\ge 1-\delta/2.
\end{equation}

By substituting into inequalities \eqref{eq:ineq_ratio_e} and $\log(1+x)>x/2$ for $x\in [0,1)$, we get the following inequality
\begin{equation}\label{eq: ineq_log_ratio}
    \P\lr{\log^{-1} \la{\frac{\si_1}{\si_2}}<2c_3^{-1}N^{2/3}}\ge 1-\delta/2.
\end{equation}

Since $\sqrt{N}\al_1$ is asymptotic Gaussian by \cite[Theorem 13]{tao2012random}, then by using the same proof as for inequality \eqref{eq: upper_b_h1}: for any $\de>0$, there exists a constant $c_4>0$ so that
\begin{equation}\label{eq: bdd_gaussian_vector}
    \lim_{N\to \iy}\P\lr{\fc{\al_1^{-2}-1}{2(1-\ep)}<c_4N}\ge 1-\de/2
\end{equation}

Thus, combining inequalities \eqref{eq: ineq_log_ratio} and \eqref{eq: bdd_gaussian_vector}, we can prove the desired result: for any $\de>0$, there exist constants $c_4, c_5=4c_3^{-1}>0$ such that 
\begin{equation}
    \lim_{N\to \infty}\P\lr{
    T_\ep< c_5N^{2/3}\log(c_4 N)
    }\ge 1-\delta.
\end{equation}

\end{proof}

%%%%%%%%%%%%%%%%%%%%%%%%%%%%%%%%%%%%%%%%%%%%%%%%%
\paragraph{Acknowledgement:}
We thanks Aukosh Jagannath, Yi Shen, and Dong Yao for their invaluable suggestions and insights.

\bibliographystyle{abbrv}
\bibliography{DSSKM}

\end{document}